\documentclass[a4paper,10pt]{article}
\usepackage[dvipsnames]{xcolor}
\usepackage{amsmath,amsfonts,amssymb,tikz,tikz-cd,geometry,enumitem,graphicx,bbm,pgfplots,subfigure,silence}
\usepackage{amsthm}
\usepackage{hyperref}

\graphicspath{{./Figures/}}
\graphicspath{ {./Figures/} }

\newtheorem{thm}{Theorem}[section]
\newtheorem*{thm*}{Theorem}
\newtheorem{cor}[thm]{Corollary}
\newtheorem*{cor*}{Corollary}
\newtheorem{lem}[thm]{Lemma}
\newtheorem{rmk}[thm]{Remark}

\newtheorem{prop}[thm]{Proposition}

\newcommand*{\defeq}{\mathrel{\vcenter{\baselineskip0.5ex \lineskiplimit0pt\hbox{\scriptsize.}\hbox{\scriptsize.}}}=}
\newcommand*{\rdefeq}{\mathrel{=\hspace{-0.1cm}\vcenter{\baselineskip0.5ex \lineskiplimit0pt\hbox{\scriptsize.}\hbox{\scriptsize.}}}}

\newcommand*{\1}{\mathbbm1}

\let\originalleft\left
\let\originalright\right
\renewcommand{\left}{\mathopen{}\mathclose\bgroup\originalleft}
\renewcommand{\right}{\aftergroup\egroup\originalright}

\renewcommand{\P}{\Para}

\let\temp\phi
\let\phi\varphi
\let\varphi\temp

\newcommand{\cB}{\mathcal{B}}

\newcommand{\E}{\mathbb{E}}

\newcommand{\N}{\mathbb{N}}
\renewcommand{\P}{\mathbb{P}}

\newcommand{\R}{\mathbb{R}}

\renewcommand{\a}{\alpha}

\newcommand{\eps}{\epsilon}

\newcommand{\Th}{\Theta}

\newcommand{\q}{\quad\quad}

\setenumerate{label={\normalfont(\roman*)}}

\numberwithin{equation}{section}

\begin{document}

\title{Error Rates for Large Deviations in the domain of an $\alpha=1$ stable law}
\author{
Jonny  Imbierski
\and
Dalia Terhesiu
\thanks{The affiliation of both authors is
Mathematisch Instituut,
University of Leiden, Einsteinweg 55,
2333, CC  Leiden, Netherlands.
Email addresses: j.f.imbierski@math.leidenuniv.nl and daliaterhesiu@gmail.com}
}
\maketitle

\begin{abstract}
    \noindent
    We obtain  error rates for large deviations of sums of i.i.d.\ random variables in, a particular case, of the domain of a non-symmetric infinite mean $\alpha=1$-stable law. The focus of this work is on the method of proof via analytic techniques rather than the particular result.
\end{abstract}

\section{Introduction}

Suppose that $(X_j)_{j\in\N}$ is a sequence of i.i.d.\ real-valued random variables on the probability space $(\Omega,\cB(\Omega),\P)$ and put $S_n\defeq X_1+\dotsb+X_n$ for each $n\in\N$.
Suppose that $(S_n-b_n)/a_n$ is in the domain of attraction of an $\a$-stable distribution, where $\alpha\in(0,2)$  and $(a_n)_{n\in\N}$ and $(b_n)_{n\in\N}$ are scaling and centring sequences given by
\begin{equation}\label{q:anbn}
    a_n=n\ell(a_n)\big(1+o(1)\big),\hspace{1cm}
    b_n=\left\{
        \begin{aligned}
            &0,                                 &&\alpha\in(0,1),\\
            &n\E(X_1\1_{\{|X_1|\leq a_n\}}),    &&\alpha=1,\\
            &n\E(X_1),                          &&\alpha\in(1,2).
        \end{aligned}\right.
\end{equation}
for some slowly varying function $\ell$.
In the i.i.d.\ setting, this is equivalent to assuming that the tails of $\P(X_1>x)$ are regularly varying with index $\alpha$, i.e.\
\begin{equation}\label{q:tailRV}
    \P(X_1>x)=px^{-\a}\ell(x)\big(1+o(1)\big)\quad\text{and}\quad\P(X_1<-x)=qx^{-\a}\ell(x)\big(1+o(1)\big),
\end{equation}
where $p,q>0, p+q=1$.
For $\a\in(0,1)\cup(1,2)$, large deviations were first given in \cite{heyde,Nag1,Nag2}. The results in these papers say that if~\eqref{q:tailRV} holds, then as $x(n)/a_n\to\infty$,
$$\P\big(S_n-b_n>x(n)\big)=pnx(n)^{-\a}\ell(x(n))\big(1+o(1)\big)\quad{and}\quad\P(S_n-b_n<-x(n))=qnx(n)^{-\a}\ell(x(n))\big(1+o(1)\big).$$
The $\a=1$ stable law was partially treated in~\cite{Nag2}, and~\cite{DDS08} and obtains the same result but only in the range $x(n)\ge \delta b_n$, for some $\delta>0$.
The case of the $\a=1$ stable law, with no restriction on the range of $x(n)$, that is in the full range $x(n)/a_n\to\infty$, was treated only recently in \cite{Berger19}. \medskip

In this work we are interested in obtaining error terms in large deviation, and here we restrict to a particular case of $\alpha=1$ with infinite mean, but to obtain error rates as in Theorem~\ref{t:LD} below, we only deal with a reduced range of $x$.
Our focus is on giving a method of proof that can provide error rates, which can be adapted to other values of $\alpha$, see~\cite{IT}.\medskip

 One simple example of interest is the case where $X_1$ is a real-valued random variable with the following  tail behaviour, for positive constants $p,q$ with $p+q=1$:
 there exists $x_0 > 0$ such that
\begin{equation}\label{q:tail}
    \P(X_1>x)=\frac{p}{x}, \quad \P(X_1<-x)=\frac{q}{x}
    \qquad \text{ for all } |x| > x_0.
\end{equation}

The current method of proof exploits heavily a precise control as $t\to 0$ of the characteristic function of $X_1$ given by $\Psi(t)\defeq\int_{-\infty}^{\infty}e^{itx}\,d\P(X_1\leq x)$.
Under assumption~\eqref{q:tail}, there is a $t_0 > 0$ such that
\begin{equation}\label{q:Psi}
\Psi(t)\big(1+o(1)\big)=1-c i|t|\log\left|t\right|+C_0 i t+C |t|, \qquad \text{ for all } |t| > t_0,
\end{equation}
for real constants $c,C_0, C$ (which, in particular, depend on $p,q$).
This precise asymptotic follow from \cite[Theorem 2]{AD98} (see also~\cite{IL}).
Given this precise asymptotic of $\Psi(t)$,
\begin{equation*}
        \Psi'(t)\big(1+o(1)\big)=-c i\log\left|t\right|+C',\q
        \Psi''(t)\big(1+o(1)\big)=-c it^{-1},\q\q\text{ for all } |t| > t_0,
\end{equation*}
for some complex constant $C'$.
We could  more generally assume that $\P(X_1>x)=\frac{p}{x}+\frac{c_1}{x^2}
+O\left(\frac{1}{x^3}\right)$ as $x\to\infty$ for some real constant $c_1$ and 
a similar behaviour for $\P(X_1<-x)$, so that $X_1$ is not symmetric. 
At the cost of some lengthy calculations one can show that as $t\to 0$,
\begin{equation}\label{q:Psi'Est}
    \begin{aligned}
        \Psi(t)&=1-c i|t|\log\left|t\right|+C_0 i t+C |t|+O(t^2\left|\log\left|t\right|\right|)\\
        \Psi'(t)    &=-ci\log\left|t\right|+ C'+O(\left|t\log\left|t\right|\right|),\quad
        \Psi''(t)   =-cit^{-1}+O(\left|\log\left|t\right|\right|),
    \end{aligned}
\end{equation}
for some real constants $c,C, C_0$ and some complex constant $C'$.
Again, the expression for $\Psi(t)$ follows from~\cite[Theorem 2]{AD98}, while the expressions of $\Psi'(t),\Psi''(t)$ require some work (as the derivatives of the $O(t^2\left|\log\left|t\right|\right|)$ term need to be estimated), which we omit.
Instead, we will phrase some assumptions in terms of the behaviour of $\Psi(t)$ as $t\to 0$.\medskip

In addition to~\eqref{q:Psi'Est}, we assume
\begin{equation}\label{q:tail1}
    \P(X_1>x)=\frac{p}{x}+O\left(\frac{1}{x^2}\right), \qquad \P(X_1<-x)=\frac{q}{x}+O\left(\frac{1}{x^2}\right).
\end{equation}
It should be noted that the precise form of the $O\left(\frac{1}{x^2}\right)$
follows from~\eqref{q:Psi'Est} and vice versa.
\medskip
%

Under~\eqref{q:tail1}, recall from~\eqref{q:anbn} that  $a_n=n(1+o(1))$ and $b_n=n\log n(1+o(1))$ as $n\to\infty$.
To state the main results, write $N=N(n)\defeq x(n)+b_n$ with $x(n)/a_n\to\infty$ and note that $$\frac N{a_n}=\frac Nn\big(1+o(1)\big)\geq\big(1+o(1)\big)\log n\to\infty,$$ as $n\to\infty$.

\begin{thm}\label{t:LD}
    Assume~\eqref{q:Psi'Est} and~\eqref{q:tail1} (so $\alpha = 1$).
    Fix a large finite number $a > 2$.
    Put $N=N(n)\defeq x(n)+b_n$ and assume that $x(n)$ is such that $n (\log N)^2 = o(N)$ as $n\to\infty$.
    Then
    \begin{align}\label{eq:mainELD}
        \P(S_n>N)=n\P(X_1>N)+O\left(\frac{n^2(\log N)^2}{N^{2}}\right)+O\left(\frac{n}{N^{2-2/a}}\right)\qquad \text{ as }n\to\infty,
    \end{align}
    and a similar statement holds for $\P(S_n<-N)$.
\end{thm}
The message of Theorem~\ref{t:LD} is that obtaining error rates for large deviations as in the setting of Theorem~\ref{t:LD} comes at the cost of reducing the range of $x(n)$ for which the result is optimal.
We do not know if this is a consequence of the current method of proof or if there is some cancellation effect that we failed to capture.\medskip

\begin{rmk}\label{rem:sym}
 If the tails in \eqref{q:tailRV} are symmetric, i.e.\ $p=q=\frac12$, then the logarithmic term in \eqref{q:Psi} and \eqref{q:Psi'Est} disappears (i.e.\ $c=0$), and also the factor $(\log N)^2$ in the first error term
 in \eqref{eq:mainELD} disappears, and the range condition simplifies to
 $n  = o(N)$. The proof of this can be deduced by following our proof with $c = 0$ throughout, but we will skip the details.\\
 
 We believe that following the scheme in~\cite{IT} one can extend Theorem~\ref{t:LD} (and possibly to other than $\alpha=1$ values) to the set up of dynamical systems (in particular, Gibbs Markov maps such as the Gauss map). However, the required work is involved and we do not consider this here.
\end{rmk}

In Section~\ref{sec:strategy}, we outline the strategy of proof along with the main ingredient, namely we rephrase Theorem~\ref{t:LD} in terms of the characteristic function $\Psi(t)$.
Section~\ref{s:pr} is allocated to deriving estimates of various quantities expressed in terms of $\Psi(t)$ and completing the proof of Theorem~\ref{t:LD}.

\section{Main ingredients of the proofs}\label{sec:strategy}

The starting point is Section~\ref{sec:char}, where we show that the statements\footnote{The statements on $\P(S_n<-N)$ follow similarly
by working with $\P(-S_n>N)$.}  on $\P(S_n>N)$ in Theorem~\ref{t:LD} follow once we obtain useful estimates for the integral
$$
I^*\defeq \frac{1}{2\pi}\int_{-\eps}^\eps\big(e^{-itN}-e^{-it(N+g(n))}\big)\frac{\Psi(t)^n-n\Psi(t)}{it}\psi_Y(t)\,dt,
$$
for some small fixed $\eps>0$, and for $g(n)$ a suitably chosen diverging sequence, and $\psi_Y(t)$ is explained in the following remark.

\begin{rmk}\label{rem:Y}
    The integral $I^*$ is derived from a Fourier inversion formula: see \eqref{q:Fourier}.
    The domain of integration is restricted to $[-\eps,\eps]$ by introducing the function $\psi_Y(t)$, which is the characteristic function of a random variable $Y$ chosen on the same probability space $(\Omega,\mathcal{B}(\Omega),\mathbb{P})$ such that $Y\in L^a$ for any $a\geq1$ and with support $\mathrm{supp}\,\psi_Y=(-\eps,\eps)$ with $\psi_Y(t)$ vanishing at $t=\pm\eps$.
    The existence of such a random variable $Y$ follows from, for instance, \cite[Proposition~3.8]{Gou-asip}.
    In fact, the function $\psi_Y(t)$ can be ensured to be $C^\infty$ although $C^2$ is sufficient for the current aim.
    The argument of restricting the domain of integration in this way goes back to \cite[Section~3, Chapter~XVI]{Feller}.
\end{rmk}
%
The translation of the problem of the rates into estimates for $I^*$ is captured in Proposition~\ref{p:gn}.
A main difficulty in expanding the integral $I^*$ in terms of $N$ and $n$ comes from the fact that the factor of $t^{-1}$ in the integrand blows up (as $t\to 0$). Since we want to let $g(n)\to\infty$ (in fact, we will require $g(n)$ much larger than $N$), we cannot exploit obtaining a $t$ from $e^{-itN}-e^{-it(N+g(n))}$.
To overcome this, we rewrite $\Psi(t)^n-n\Psi(t)$ as 
\begin{align}\label{q:PsiRewrite}
    \Psi(t)^n-n\Psi(t)
        &=\big(\Psi(t)^{n-1}-1\big)\Psi(t)-(n-1)\Psi(t)=\big(\Psi(t)-1\big)\sum_{k=1}^{n-1}\Psi(t)^k-(n-1)\Psi(t)\nonumber\\
        &=\big(\Psi(t)-1\big)\sum_{k=1}^{n-1}\big(\Psi(t)^k-1\big)+(n-1)\big(\Psi(t)-1\big)-(n-1)\Psi(t)\nonumber\\
        &=\big(\Psi(t)-1\big)\sum_{k=1}^{n-1}\big(\Psi(t)^k-1\big)-(n-1).
\end{align}
Due to the first equality in our assumption~\eqref{q:Psi'Est}, the effect of $t^{-1}$ will be compensated by $(\Psi(t)-1)$.
In order to expand $I^*$ in terms of $N$ and $n$,
we will also use integration by parts (and thus, derivatives of $\Psi$) and a modulus of continuity argument as in \cite[Chapter 1]{Katzn}.
The calculations are not trivial and require several technical steps, but the strategy is as described in this paragraph.

\subsection{Rephrasing the Problem in Terms of the Characteristic Function}\label{sec:char}

We translate the statement of Theorem~\ref{t:LD} into a statement on the characteristic function $\Psi(t)$ by using the method in \cite[Section 2]{IT}, but recording the error terms along the way.
The results below are analogues of \cite[Lemma 2.1]{IT} and \cite[Proposition 2.2]{IT} tailored to Theorem~\ref{t:LD}.

\begin{lem}\label{lemY}
    Let $(Z_n)_{n \in \N}$ be a sequence of random variables and $Y$ another random variable, defined on the same probability space but independent of all the $Z_n$.
    Assume that there exist three sequences $N=N(n)$, $z_N=z_{N(n)}$ and $y_N=y_{N(n)}$ such that, as $n\to\infty$, it holds that $N\to\infty$, $z_N\to\infty$ with $z_N=o(N)$, $y_N\to 0$ and
    \begin{itemize}
        \item[\em(a)] $y_N=o\left(\P(Z_n > N)\right)$ and $\P(Z_n > N \pm z_N) = \P(Z_n > N)+ O(y_N)$.
        \item[\em(b)]$\P(|Y| > z_N) = O(y_N)$.
    \end{itemize}
    Then $\P(Z_n+Y> N) = \P(Z_n > N)+ O(y_N)$ as $n \to \infty$.
\end{lem}

\begin{proof}
Since $Z_n$ and $Y$ are independent,
\begin{align*}
    \P(Z_n > N+z_N)
        &=\P(Z_n > N+z_N,\, |Y| > z_N) +  \P(Z_n > N+z_N ,\, |Y| \leq z_N)\\
        &\leq\P(|Y| > z_N) +  \P(Z_n+Y > N).
\end{align*}
Therefore, using (a) and (b), $$\P(Z_n+Y > N) \geq \P(Z_n > N+z_N) - \P(|Y| > z_N) = \P(Z_n > N)+O(y_N).$$
For the other inequality, we have
\begin{align*}
    \P(Z_n \leq N-z_N)
        &=\P(Z_n \leq N-z_N ,\, |Y| > z_N) +  \P(Z_n \leq N-z_N,\, |Y| \leq z_N)\\
        &\leq\P(|Y| > z_N)+\P(Z_n+Y \leq N).
\end{align*}
Therefore,
\begin{eqnarray*}
\P(Z_n+Y > N) &=& 1-\P(Z_n+Y \leq N) \leq 1-\P(Z_n \leq N-z_N) + \P(|Y| > z_N) \\
&=& \P(Z_n > N-z_N) + \P(|Y| > z_N) = \P(Z_n > N)+O(y_N),
\end{eqnarray*}
where in the last equality we have used (a) and (b).
This finishes the proof.
\end{proof}

The next proposition is tailored to the statement of Theorem~\ref{t:LD}, but as in~\cite[Section 2]{IT} it can be adjusted to other rates and other $\alpha\ne 1$ stable laws.

\begin{prop}\label{p:gn}
    Assume the set up of Theorem~\ref{t:LD}.
    Let $a>2$ be a very large positive finite number.
    Let $Y$ be as in remark~\ref{rem:Y} and
    let $g(n)$ be so that $g(n)\geq N^2$.
    Then equation~\eqref{eq:mainELD} holds if and only if 
    \begin{equation}\label{eq:right}
         \left|\int_{-\eps}^\eps\left(e^{-itN}-e^{-it(N+g(n))}\right)\psi_Y(t)\,\frac{\Psi(t)^n-n\Psi(t)}{it}\, dt\right|
            =O\left(\frac{n}{N^{2-2/a}}\right)
    \end{equation}
    as $n\to\infty$.
\end{prop}

\begin{proof} Recall that $X_1$ and $S_n$ are random variables taking values in $\R$.
So, $\P(X_1=N)=\P(X_1=N+g(n))=0$ and $\P(S_n=N)=\P(S_n=N+g(n))=0$.
     For a fixed $N\in\R$, we start from the inversion formulae
    \begin{equation}\label{q:Fourier}
        \begin{aligned}
            \P\big(X_1 \in (N, N+g(n)]\big) &=\lim_{T\to\infty}\frac{1}{2\pi }\int_{-T}^T\frac{e^{-itN}-e^{-it(N+g(n))}}{it}\Psi(t)\, dt,\q\text{and}\\ 
            \P\big(S_n\in (N, N+g(n)\big)   &=\lim_{T\to\infty}\frac{1}{2\pi }\int_{-T}^T\frac{e^{-itN}-e^{-it(N+g(n))}}{it}\Psi(t)^n\, dt,
        \end{aligned}
    \end{equation}
    for any choice of $g(n)$.\medskip
    
    We want to allow that $N(n),\, g(n)\to\infty$ so as to obtain the desired formula for $\P(S_n > N)$.
    By our assumption $g(n)\gg N^2$,
    \[\P(X_1 > N+g(n))\ll \frac{1}{g(n)} \frac{1}{\left(1+\frac{N}{g(n)}\right)} \ll \frac{1}{N^{2}}\] and
    \[\P(S_n > N+g(n)) \le \sum_{j=0}^{n-1} \P\left(|X_j|>\frac{N+g(n)}{n}\right) \ll \frac{n^2}{N^2}.\]
    Therefore, as $n\to\infty$, 
    \begin{equation}\label{eq:mform}
        \P(S_n>N)=\lim_{T\to\infty}\frac{1}{2\pi }\int_{-T}^T\frac{e^{-itN}-e^{-it(N+g(n))}}{it}\Psi(t)^n\, dt+ O\left(\frac{n^2}{N^2}\right)
    \end{equation}
    and
    \begin{equation}\label{eq:mformx1}
        \P(X_1 >N)=\lim_{T\to\infty}\frac{1}{2\pi }\int_{-T}^T\frac{e^{-itN}-e^{-it(N+g(n))}}{it}\Psi(t)\, dt + O\left(\frac{1}{N^2}\right).
    \end{equation}
    
    Next, we argue that one can adjust the domain of integration in~\eqref{eq:mform}.
     Let $Y$ be as in Remark~\ref{rem:Y}.
    Since $Y\in L^a$ and $\psi_Y$ is supported on $(-\eps,\eps)$, a repetition of the previous arguments gives that the analogues of~\eqref{eq:mformx1} and~\eqref{eq:mform} are:
    \begin{equation}
    \label{eq:mformx1Y}
     \P(X_1+Y>N)=\frac{1}{2\pi }\int_{-\eps}^\eps\frac{e^{-itN}-e^{-it(N+g(n))}}{it}\psi_Y(t)\Psi(t)\, dt+O\left(\frac{1}{N^2}\right).
    \end{equation}
    and
    \begin{equation}\label{eq:mformY}
     \P(S_n+Y>N)=\frac{1}{2\pi}\int_{-\eps}^\eps\frac{e^{-itN}-e^{-it(N+g(n))}}{it}\psi_Y(t) \Psi(t)^n\, dt+O\left(\frac{n^2}{N^2}\right).
    \end{equation}
    
    To conclude we need to compare the tails of $X_1$ and $X_1+Y$, and of $S_n$ and $S_n+Y$.
    To do so, we will use Lemma~\ref{lemY} several times, making different choices for $Z_n$ appearing in that lemma.
    We will first apply Lemma~\ref{lemY} to $Z_n\equiv X_1$ with $a > 2$, $N=N(n)$ as in the statement of the current proposition and
    \begin{itemize}
        \item $z_N=z_{N(n)}=N^{2/a}=o(N)$.
        \item $y_N=y_{N(n)}=N^{-(2-2/a)}\to 0$. 
    \end{itemize}
    
    Recall that $Y\in L^a$. With the above choice of $z_N$,
    we see that $\P(Y>z_N)=o(z_N^{-a})=o(N^{-2})=o(y_N)$.
    Thus, we can apply Lemma~\ref{lemY} with $Z_n \equiv X_1$ since both (a) and (b) hold.
    This gives
    \begin{equation}\label{eqPPY}
        \P(X_1>N) = \P(X_1+Y > N)+O(y_N)=\frac{1}{2\pi }\int_{-\eps}^\eps\frac{e^{-itN}-e^{-it(N+g(n))}}{it}\psi_Y(t)\Psi(t)\, dt+O(y_N),
    \end{equation}
    where in the last equality we have used~\eqref{eq:mformx1Y}.
    Next we want to obtain a version with $S_n$ in~\eqref{eq:mformY}.\medskip
    
    \textbf{From~\eqref{eq:mainELD} to~\eqref{eq:right}}.
    If $S_n$ satisfies~\eqref{eq:mainELD}, i.e.\\ $\P(S_n>N) = n\P(X_1 > N)+O(y_N)$, then the above application of Lemma~\ref{lemY} to $X_1$ yields $\P(S_n > N+z_N) = \P(S_n > N)+O(ny_N)$.
    We already know that $\P(|Y| > z_N)=o(y_N)$ so we may apply Lemma~\ref{lemY} with $Z_n = S_n$, $z_N$ and $ny_N$ to conclude that $$\P(S_n>N) = \P(S_n+Y>N)+O(ny_N).$$
    Equation~\eqref{eq:right} follows by combining~\eqref{eq:mformx1Y}
    and~\eqref{eq:mformY}.\medskip
    
    \textbf{From~\eqref{eq:right} to~\eqref{eq:mainELD}}.
    If~\eqref{eq:right} holds, then using~\eqref{eq:mformY} and  \eqref{eqPPY}, we obtain
    $$\P(S_n+Y > N) = n \P(X_1 > N)+O\left(ny_N\right).$$
    Recall that $\P(|Y| > z_N) =o(y_N)$ so we may apply Lemma~\ref{lemY} with $Z_n = S_n+Y$ to conclude that
    \begin{equation}\label{eqSSP}
    \P(S_n+Y > N) = \P(S_n > N)+O(ny_N).
    \end{equation}
    Hence, after subtracting \eqref{eq:mformx1Y} $n$ times
    from~\eqref{eq:mformY}  and applying~\eqref{eqPPY} and~\eqref{eqSSP},
    we obtain that
    \begin{align*}
        \left|\P(S_n>N)-n \P(X_1>N)\right|
            &=\frac{1}{2\pi} \left|\int_{-\eps}^\eps (e^{-itN}-e^{-it(N+ g(n))})\psi_Y(t)\,\frac{\Psi(t)^n-n\Psi(t)}{it} \, dt\right|+O(ny_N)\\
            &= O\left(ny_N\right) = O\left( \frac{n}{ N^{2-2/a}}\right).\qedhere
    \end{align*}
\end{proof}

\subsection{Simplifying the expression of the integral in Proposition~\ref{p:gn} via Equation~\eqref{q:PsiRewrite} }

By Proposition~\ref{p:gn}, it suffices to obtain useful estimates (in $N=N(n)$ and $n$)
of
$$|I^*|=\left| \frac{1}{2\pi} \int_{-\eps}^\eps (e^{-itN}-e^{-it(N+ g(n))})
 \psi_Y(t)\,\frac{\Psi(t)^n-n\Psi(t)}{it} \, dt\right|.$$
 By~\eqref{q:PsiRewrite}, $
    \Psi(t)^n-n\Psi(t)
        =\big(\Psi(t)-1\big)\sum_{k=1}^{n-1}\big(\Psi(t)^k-1\big)-(n-1)$.
Set
\begin{equation}\label{q:FDef}
    F(t,n)\defeq\sum_{k=1}^{n-1}\big(\Psi(t)^k-1\big)
\end{equation}
 and note that
\begin{align*}
    \nonumber I^*&=\frac{1}{2\pi} \int_{-\eps}^\eps\frac{e^{-itN}-e^{-it(N+g(n))}}{it}\psi_Y(t)\big(\Psi(t)-1\big)F(t,n)\,dt-(n-1)\int_{-\eps}^\eps\frac{e^{-itN}-e^{-it(N+g(n))}}{it}\psi_Y(t)\,dt.
\end{align*}
But $\frac{(n-1)}{2\pi}\int_{-\eps}^\eps\frac{e^{-itN}-e^{-it(N+g(n))}}{it}\psi_Y(t)\,dt
=(n-1)\P\big(N<Y<N+g(n)\big)\leq n\P(Y>N)=O(n/N^2)$.
So, 
\begin{equation}\label{q:IDef}
    I^*=\frac{1}{2\pi} \int_{-\eps}^\eps\big(e^{-itN}-e^{-it(N+g(n))}\big)\frac{\Psi(t)-1}{it}\psi_Y(t)F(t,n)\,dt+O\left(\frac n{N^2}\right)\defeq I+O\left(\frac n{N^2}\right).
\end{equation}

Putting all these together, we obtain the following.

\begin{cor}\label{cor:I}
    Let $I$ be as in~\eqref{q:IDef}, and fix $a>2$ large. 
    Let $g(n)\ge N^2$.
    Then the conclusion of Theorem~\ref{t:LD} follows if $|I|=O((n\log N)^2/N^2)+O\left(n/N^{2-2/a}\right)$.
\end{cor}

\begin{proof}
     This follows from Proposition~\ref{p:gn} and equation~\eqref{q:IDef}. 
\end{proof}

\section{Proof of Theorem~\ref{t:LD}}\label{s:pr}

\subsection{ Estimates for $F(t,n)$ defined in~\eqref{q:FDef} along with its derivatives and related quantities}\label{sec:Est}

Using Corollary~\ref{cor:I} and the definition of $I$ in~\eqref{q:IDef},
the main task from here onward is to obtain useful estimates for $F(t,n)$ 
to be used later during the integration by parts and modulus of continuity arguments (see Section~\ref{sec:I3} for details of this type of argument).
For use in the integration by parts, we need to estimate $F'$, while for use in modulus of continuity we will need an estimate for $F'(t,n)-F'(t-h,n)$, $0<h<\left|t\right|<\eps$, such that every term including a logarithm is explicit (see \eqref{q:F'Cont} below).\medskip

We start by recalling our assumption~\eqref{q:Psi'Est}
for real constants $c, C, C_0$ and complex constant $C'$,
as $t\to 0$:

    \begin{align}
           \label{eq:psiat0}  \Psi(t)&=1-c i|t|\log\left|t\right|+C_0 i t+C |t|+A(t),\text{ where } A(t)=O(t^2\left|\log\left|t\right|\right|),\\
        \label{eq:prime}\Psi'(t)    &=-ci\log\left|t\right|+ C'+ A'(t),\text{ where }A'(t)=O(\left|\log\left|t\right|\right|),\\
        \label{eq:second}\Psi''(t)   &=-ci|t|^{-1}+A''(t),\text{ where } A''(t)=O(\left|\log\left|t\right|\right|).
    \end{align}

Using~\eqref{eq:psiat0},
\begin{equation}\label{q:FEst}
    F(t,n)=\big(\Psi(t)-1\big)\sum_{j=1}^{n-1}\sum_{k=0}^{j-1}\Psi(t)^k=O(n^2\bigl|t\log\left|t\right|\bigr|).
\end{equation}

Differentiating the expression for $F(t,n)$ in \eqref{q:FDef} and applying \eqref{eq:prime}, we obtain
\begin{equation}\label{q:F'Est}
    F'(t,n)=\Psi'(t)\sum_{j=1}^{n-1}j\Psi(t)^{j-1}
        =\bigl(-ci\log\left|t\right|+C'\bigr)\sum_{j=1}^{n-1}j\Psi(t)^{j-1}+O\left(n^2\bigl|t\log\left|t\right|\bigr|\right).
\end{equation}

We are now interested in getting an estimate of $F'(t,n)-F'(t-h,n)$, $0<h<\left|t\right|<\eps$.
By~\eqref{eq:prime}, the term $A(t)$ in~\eqref{eq:psiat0} is differentiable.
By the mean value theorem, the following holds for some $t_0\in (t-h, t)$: $|A(t)-A(t-h)| \leq h |A'(t_0)|$. Thus, $|A(t)-A(t-h)|\ll h A'(t)=O(h|t\log\left|t\right|)$.
Therefore,
\begin{align}\label{q:PsiCont}
    \Psi(t)-\Psi(t-h)
        &=ci|t|\log\left|1-\frac ht\right|\pm ci h\log\left|t\right|+ C_1 h+O(h|t\log\left|t\right|),
\end{align}
for a complex constant $C_1$.\medskip   

By a similar argument, using~\eqref{eq:second} and the mean value theorem,
$|A'(t)-A'(t-h)|\ll h |A''(t)|=O(h|\log\left|t\right|)$.
This together with~\eqref{eq:prime} gives
\begin{align}\label{q:Psi'Cont}
    \Psi'(t)-\Psi'(t-h)
        &=ci\log\left|1-\frac ht\right|+O(h|\log\left|t\right|).
\end{align}
 Continuing from the first equality in~\eqref{q:F'Est},
and using
  \eqref{q:Psi'Cont} and~\eqref{q:PsiCont} together with~\eqref{eq:prime}, 
\begin{align*}
    &F'(t,n)-F'(t-h,n)\nonumber\\
    &\q=\big(\Psi'(t)-\Psi'(t-h)\big)\sum_{j=1}^{n-1}j\Psi(t)^{j-1}+\Psi'(t-h)\sum_{j=1}^{n-1}j\big(\Psi(t)^{j-1}-\Psi(t-h)^{j-1}\big)\nonumber\\
    &\q=\left(ci\log\left|1-\frac ht\right|+O(h\log\left|t\right|)\right)\sum_{j=1}^{n-1}j\Psi(t)^{j-1}\\
    &\q\q-ci\left(cit\log\left|1-\frac ht\right|+h C_1\right)\log\left|t-h\right|\sum_{j=1}^{n-1}\sum_{k=0}^{j-2}j\Psi(t)^k\Psi(t-h)^{j-2-k}\big)\nonumber\\
    &\q\q+O\left(t^2\log\left|1-\frac ht\right|\sum_{j=1}^{n-1}j^2\left|\Psi(t)\right|^{j-1}\right).\nonumber
\end{align*}
Recall that $h<t$. Thus, as $h\to 0$ and $t\to 0$, $\log\left|1-\frac ht\right|=h|t|^{-1}+O(h|t|^{-1})=h|t|^{-1}(1+o(1))$. Thus,
\[
 \left|\left(cit\log\left|1-\frac ht\right|+h C_1\right)\log\left|t-h\right|\sum_{j=1}^{n-1}\sum_{k=0}^{j-2}j\Psi(t)^k\Psi(t-h)^{j-2-k}\right|=O\left(h\sum_{j=1}^{n}j^2\log\left|t\right||\Psi(t)|^{j}\right)
\]
and
\begin{align}\label{q:F'Cont}
    F'(t,n)-F'(t-h,n)
    &=ci\log\left|1-\frac ht\right|\sum_{j=1}^{n-1}j\Psi(t)^{j-1}
    +O\left(h\sum_{j=1}^{n}j^2\log\left|t\right||\Psi(t)|^{j}\right)\nonumber\\
    &=cih|t|^{-1}(1+o(1))\sum_{j=1}^{n-1}j\Psi(t)^{j-1}+O\left(h\sum_{j=1}^{n}j^2\log\left|t\right||\Psi(t)|^{j}\right).
\end{align}

In Section~\ref{sec:MainIntegral}, we will work with the function $\Th(t)\defeq(\Psi(t)-1)/it$ and its derivatives.
Dividing~\eqref{eq:psiat0} by $it$, we get
\begin{align}\label{q:ThEst}
    \Th(t)=-c\log\left|t\right|+C_2+O\left(\bigl|t\log\left|t\right|\bigr|\right),
    \text{ where } C_2\text{ is a complex constant}.
\end{align}

Compute that $\Th'(t)=\frac{\Psi'(t)}{it}-\frac{\Psi(t)-1}{it^2}$
and that $\Th''(t)=\frac{\Psi''(t)}{it}-2\frac{\Psi'(t)}{it^2}-2\frac{\Psi(t)-1}{it^3}$. Using~\eqref{eq:prime},~\eqref{eq:second} and~\eqref{eq:psiat0} and taking advantage of the cancellations, we obtain that
\begin{equation}\label{q:Th'Est}
    \Th'(t)=-c|t|^{-1}+ O\left(\left|\log\left|t\right|\right|\right)\quad\text{and}\quad
    \Th''(t)=c|t|^{-2}+O(\left|t\right|^{-1}\log\left|t\right|),
\end{equation}

By~\eqref{q:ThEst}, $\Th(t)-\Th(t-h)=-c\log\left|1-\frac ht\right|+B(t)-B(t-h)$,
where $B(t)=O\left(\bigl|t\log\left|t\right|\bigr|\right)$. By the mean value theorem, $|B(t)-B(t-h)| \leq h|B'(t_0)|$ for $t_0\in [t-h, t]$.
Together with the first equality in~\eqref{q:Th'Est},
$|B(t)-B(t-h)|\ll h\log\left|t\right|$.
Thus,
\begin{align}\label{q:ThCont}
    \Th(t)-\Th(t-h)
        =ch|t|^{-1}+O\left(h\log\left|t\right|\right).
\end{align}

\subsection{Estimating the integral $I$ defined in~\eqref{q:IDef}}\label{sec:MainIntegral}

Recall that $\Th(t)\defeq(\Psi(t)-1)/it$ and $F(t,n)\defeq\sum_{j=1}^{n-1}(\Psi(t)^j-1)$.
Given $I$ defined in~\eqref{q:IDef}, we will write $I=I(N)-I(N+g(n))$, where
\begin{equation}\label{q:INDef}
    I(N)\defeq\int_{-\eps}^\eps e^{-itN}\Th(t)\psi_Y(t)F(t,n)\,dt\quad\text{and}\quad I\big(N+g(n)\big)\defeq\int_{-\eps}^\eps e^{-it(N+g(n))}\Th(t)\psi_Y(t)F(t,n)\,dt.
\end{equation}

\begin{prop}\label{p:IExp}
    As $n\to\infty$,
    \begin{align*}
        I (N) =O\left(\frac{n^2(\log N)^2}{N^2}\right),\quad I(N+g(n))=O\left(\frac{n^2(\log (N+g(n)))^2}{(N+g(n))^2}\right).
    \end{align*}
\end{prop}

We will provide the argument for $I(N)$. The argument for $I(N+g(n))$ is the same.

Recall that $\psi_Y$ is supported in $(-\eps,\eps)$, so $\psi_Y(\eps)=\psi_Y(-\eps)
=0$.
Integrating $I(N)$ by parts gives
\begin{align}\label{q:ITerms}
    I(N)
        &=\frac1{iN}\int_{-\eps}^\eps e^{-itN}\Th(t)\psi_Y'(t)F(t,n)\,dt\nonumber\\
        &\q+\frac1{iN}\int_{-\eps}^\eps e^{-itN}\Th'(t)\psi_Y(t)F(t,n)\,dt+\frac1{iN}\int_{-\eps}^\eps e^{-itN}\Th(t)\psi_Y(t)F'(t,n)\,dt\nonumber\\
        &\rdefeq I_1(N)+I_2(N)+I_3(N).
\end{align}

The treatment of $I_1(N)$ is simple and included in Lemmas~\ref{l:I1} below.
The integrals $I_2(N)$ and $I_3(N)$ are less straightforward 
and they will be treated in Subsections~\ref{sec:I2} and \ref{sec:I3} respectively.

\begin{lem}\label{l:I1}
    As $n\to\infty$,
    $I_1(N)=O(n^2/N^2)$.
\end{lem}

\begin{proof} Recall that $\psi_Y$ is $C^2$.
    Integrating by parts gives
    \begin{align*}
        I_1(N)
            &=\left[\frac{e^{-itN}}{N^2}\Th(t)\psi_Y'(t)F(t,n)\right]_{-\eps}^\eps-\frac1{N^2}\int_{-\eps}^\eps e^{-itN}\Th'(t)\psi_Y'(t)F(t,n)\,dt\\
            &\q-\frac1{N^2}\int_{-\eps}^\eps e^{-itN}\Th(t)\psi_Y''(t)F(t,n)\,dt-\frac1{N^2}\int_{-\eps}^\eps e^{-itN}\Th(t)\psi_Y'(t)F'(t,n)\,dt\\
            &=O\left(\frac{n^2}{N^2}\right)+O\left(\frac{n^2}{N^2}\int_0^\eps\left|\log t\right|\,dt+\frac{n^2}{N^2}\int_0^\eps t(\log t)^2\,dt
                +\frac{n^2}{N^2}\int_0^\eps(\log t)^2\,dt\right)=O\left(\frac{n^2}{N^2}\right),
    \end{align*}
    where in the last line, to estimate the first three terms  we have used~\eqref{q:FEst},~\eqref{q:Th'Est} together with~\eqref{q:ThEst}, while for the fourth terms we have used~\eqref{q:F'Est} and ~\eqref{q:ThEst}.
\end{proof}

\subsubsection{Estimating the integral $I_2$ defined in~\eqref{q:ITerms}}\label{sec:I2}

We first record a lemma that will be used to estimate various integrals, in particular
in dealing with $I_2$. 

\begin{lem}\label{lem:tec}
    Let $m, r\in\N$.
    As $M\to\infty$, $$\int_{\pi/M}^\eps e^{-it M} t^{-1}(\log t)^r\psi_Y(t)\Psi(t)^m\,dt=d_0(\log M)^r\big(1+o(1)\big)$$ for a constant $d_0$ depending on $m$ and $r$.
\end{lem}

\begin{proof}
    Write $D_R(t)$ (resp.\ $D_I(t)$) for the real (resp.\ imaginary) part of $D(t)\defeq\big(\frac{\log t}{\log M}\big)^r\psi_Y(t)\Psi(t)^m$.
    Via the change of variables $\sigma=t M$,
    \begin{align}\label{q:Oscillatory}
        \int_{\pi/M}^\eps e^{-it M}t^{-1}(\log t)^r\psi_Y(t)\Psi(t)^m\,dt
            &=(\log M)^r\int_{\pi}^{\eps M}e^{-i\sigma}\sigma^{-1}\left(\frac{\log(\sigma/M)}{\log M}\right)^r
                \psi_Y\left(\frac \sigma M\right)\Psi\left(\frac\sigma M\right)^m\,d\sigma\nonumber\\
            &=(\log M)^r\int_{\pi}^{\eps M}\frac{\cos\sigma+i\sin\sigma}{\sigma}\left(D_R\left(\frac \sigma M\right)+iD_I\left(\frac\sigma M\right)\right)\,d\sigma.
    \end{align}
    We know that $\left|D(\sigma/M)\right|\leq 1$ and in fact converges to $1$ as $M\to\infty$.
    In addition, $D_R(t)$ is differentiable in $t$ although $D_R'(t) \approx (\log t)^{r-1}t^{-1}(\log M)^{-r}$ as $t \to 0$ by \eqref{eq:prime}.
    Hence, for each $\sigma \in [\pi,\eps M]$, by the mean value theorem, there is $\hat \sigma \in [\sigma,\sigma+\pi]$ such that
    \begin{equation}\label{q:D}
        D_R\left( \frac{\sigma+\pi}{M} \right) - D_R\left( \frac{\sigma}{M} \right)
            = \frac{\pi}{M} D_R'\left(\frac{\hat \sigma}{M} \right) = O\left(\frac1{\sigma\log M}\right) \qquad \text{ as } M \to \infty.
    \end{equation}
    Without loss of generality, we can assume that $\eps M = 2\pi K + \pi$ is an odd multiple of $\pi$.
    Thus the oscillatory integral $\int_{\pi}^{\eps M} \frac{\cos \sigma}{\sigma} D_R(\sigma/M) \, d\sigma$ can be written as
    $\sum_{k=1}^K J_k$ for
    \begin{align*}
      J_k
        &\defeq\int_{2k\pi - \pi}^{2k\pi + \pi} \frac{\cos \sigma}{\sigma} D_R\left(\frac\sigma M\right)\,d\sigma
            =\int_{2k\pi - \pi}^{2k\pi} \frac{\cos \sigma}{\sigma} D_R\left(\frac\sigma M\right)\, d\sigma  + \int_{2k\pi}^{2k\pi + \pi} \frac{\cos \sigma}{\sigma} D_R\left(\frac\sigma M\right) \, d\sigma \\
        &\,=\int_{2k\pi - \pi}^{2k\pi}\left(\frac{\cos\sigma}{\sigma}D_R\left(\frac\sigma M\right)
            +\frac{ \cos(\sigma+\pi)}{\sigma+\pi}D_R\left(\frac{\sigma+\pi} M\right)\right)\,d\sigma\\
        &\,=\int_{2k\pi - \pi}^{2k\pi}\left(\frac{\pi\cos \sigma}{\sigma(\sigma+\pi)}D_R\left(\frac\sigma M\right)
            +O\left(\frac{ \cos(\sigma+\pi)}{\sigma(\sigma+\pi)}\frac1{\log M}\right)  \right)\, d\sigma=O\left(\frac1{ k(k+1)}\right),
    \end{align*}
    where we have added and subtracted terms and then applied \eqref{q:D} in the last line.
    This means that the sum and thus the oscillatory integral converges:
    $$\int_{\pi}^{\eps M} \frac{\cos \sigma}{\sigma} D_R\left(\frac\sigma M\right)\,d\sigma=\sum_{k=1}^K J_k=O(1).$$
    The arguments for the other oscillatory integrals in \eqref{q:Oscillatory} involving $\sin \sigma$ and/or $D_I(\sigma/M)$ are the same, which concludes the proof.
\end{proof}

We recall that $N=N(n)$ is a function of $n$.
\begin{lem}\label{l:I2}
    As $n\to\infty$,
    $I_2(N)=O\left(n^2(\log N)/N^2\right).
    $
\end{lem}

\begin{proof} From here onward we let $h=\pi/N$ and write
    \begin{align}\label{eq:i2}
        I_2(N)
            &=\frac1{iN}\int_{[-\eps,-h]\cup[h,\eps]}e^{-itN}\Th'(t)\psi_Y(t)F(t,n)\,dt+\frac1{iN}\int_{-h}^h e^{-itN}\Th'(t)\psi_Y(t)F(t,n)\,dt\nonumber\\
            &=I_2^+(N)+I_2^{(0)}(N).
    \end{align}

    For $I_2^{(0)}$ we will take advantage of the range of the integration, namely $[-h, h]=[-\frac{\pi}{N},\frac{\pi}{N}]$, while
    for $I_2^+$ we will use integration by parts.\medskip
    
    \underline{Estimating $I_2^{(0)}$}\medskip

    By \eqref{eq:psiat0} and \eqref{q:Th'Est},
    \begin{align*}
        \big(\Psi(t)-1\big)\Th'(t)
            &=\bigl(-ci|t|\log\left|t\right|+O(\left|t\right|)\bigr)\left(-c|t|^{-1}+O\left(\left|\log\left|t\right|\right|\right)\right)
            =c^2i\log\left|t\right|+O(1). 
    \end{align*}
    This together with the first equality in~\eqref{q:FEst},
    \begin{align*}
        I_2^{(0)}(N)
            &
            =\frac1{iN}\sum_{j=1}^{n-1}\sum_{k=0}^{j-1}\int_{-h}^h e^{-itN}\big(\Psi(t)-1\big)\Th'(t)\psi_Y(t)\Psi(t)^k\,dt\\
            &\,=\frac{c^2}N\sum_{j=1}^{n-1}\sum_{k=0}^{j-1}\int_{-h}^h e^{-itN}\log\left|t\right|\psi_Y(t)\Psi(t)^k\,dt
                +O\left(\frac{n^2}{N^2}\right).
    \end{align*}
    Recall $h=\frac{\pi}{ N}$.
    Then
    \begin{align*}
        \left|\int_{-\pi/N}^{\pi/N} e^{-itN}\log\left|t\right|\psi_Y(t)\Psi(t)^k\,dt\right|\ll \int_{-\pi/N}^{\pi/N}\log\left|t\right|\, dt\ll \frac{\log N}{N}.
    \end{align*}
    Thus,
    \begin{align}\label{eq:i20}
     |I_2^{(0)}(N)|\ll \frac{n^2\log N}{N^2}.
    \end{align}

    \underline{Estimating $I_2^{+}$}\medskip
    
    As already mentioned, here we will integrate by parts.
    
    \begin{align*}
         I_2^+(N)
        &=\left[\frac{e^{-itN}}{N^2}\Th'(t)\psi_Y(t)F(t,n)\right]_{-\pi/N}^{\pi/N}
        +\left[\frac{e^{-itN}}{N^2}\Th'(t)\psi_Y(t)F(t,n)\right]_{-\eps}^{\eps}\\
        \nonumber&-\frac1{N^2}\int_{[-\eps,-h]\cup[h,\eps]}e^{-itN}\Th''(t)\psi_Y(t)F(t,n)\,dt-\frac1{N^2}\int_{[-\eps,-h]\cup[h,\eps]}e^{-itN}\Th'(t)\psi_Y'(t)F(t,n)\,dt\\
        &\q-\frac1{N^2}\int_{[-\eps,-h]\cup[h,\eps]}e^{-itN}\Th'(t)\psi_Y(t)F'(t,n)\,dt.\nonumber
    \end{align*}

    Recall that $\psi_Y$ is supported on $(-\eps,\eps)$, so $\psi_Y(\eps)=\psi_Y(-\eps)
=0$. Thus, $\left[\frac{e^{-itN}}{N^2}\Th'(t)\psi_Y(t)F(t,n)\right]_{-\eps}^{\eps}=0$.
Using the same estimates as in dealing with $I_2^{(0)}(N)$ above, we obtain
\[
\left| \left[\frac{e^{-itN}}{N^2}\Th'(t)\psi_Y(t)F(t,n)\right]_{-\pi/N}^{\pi/N}\right|
\ll \frac{n^2\log N}{N^2}.
\]
So,
    
    \begin{align}\label{q:I2R1Terms}
        \nonumber I_2^+(N)=
        &-\frac1{N^2}\int_{[-\eps,-h]\cup[h,\eps]}e^{-itN}\Th''(t)\psi_Y(t)F(t,n)\,dt-\frac1{N^2}\int_{[-\eps,-h]\cup[h,\eps]}e^{-itN}\Th'(t)\psi_Y'(t)F(t,n)\,dt\\
        &\q-\frac1{N^2}\int_{[-\eps,-h]\cup[h,\eps]}e^{-itN}\Th'(t)\psi_Y(t)F'(t,n)\,dt+O\left(\frac{n^2\log N}{N^2}\right)\nonumber\\
       & =A_1+A_2+A_3+O\left(\frac{n^2\log N}{N^2}\right).
    \end{align}

    For the term $A_1$ in~\eqref{q:I2R1Terms}, by first applying the second equality in \eqref{q:Th'Est} and then using the first equality in~\eqref{q:FEst} together with~\eqref{eq:psiat0} (where we combine $C_0 i$ and $C$ into another constant $C_2$), 
    \begin{align*}
        A_1 &=-\frac{1}{N^2}\int_{[-\eps,-h]\cup[h,\eps]}{e^{-itN}\Th''(t)\psi_Y(t)F(t,n)}\,dt\\
            &=-\frac{c}{N^2}\int_{[-\eps,-h]\cup[h,\eps]}{e^{-itN}t^{-2}\psi_Y(t)F(t,n)}\,dt
                +O\left(\frac{n^2}{N^2}\int_{[-\eps,-h]\cup[h,\eps]}(\log\left|t\right|)^2\right)\\
            &=\frac{c^2i}{N^2}\sum_{j=1}^{n-1}\sum_{k=0}^{j-1}\int_{[-\eps,-h]\cup[h,\eps]}e^{-itN}|t|^{-1}\log\left|t\right|\psi_Y(t)\Psi(t)^k\,dt\\
            &\q-\frac{C_2}{N^2}\sum_{j=1}^{n-1}\sum_{k=0}^{j-1}\int_{[-\eps,-h]\cup[h,\eps]}e^{-itN}|t|^{-1}\psi_Y(t)\Psi(t)^k\,dt
                + O\left(\frac{n^2}{N^2}\right).
    \end{align*}

    We continue with $A_2$ in~\eqref{q:I2R1Terms}.
    By~\eqref{q:Th'Est} and~\eqref{q:FEst},
    \begin{align*}
        A_2=-\frac{1}{N^2}\int_{[-\eps,-h]\cup[h,\eps]}e^{-itN}\Th'(t)\psi_Y'(t)F(t,n)\,dt=O\left(\frac{n^2}{N^2}\int_{h}^{\eps}\left|\log t\right|\,dt\right)=O\left(\frac{n^2}{N^2}\right).
    \end{align*}
    
    Next, we look at the term $A_3$ in~\eqref{q:I2R1Terms}.
    By \eqref{eq:prime} and \eqref{q:Th'Est}, we find that
    \begin{align*}
        \Th'(t)\Psi'(t)
            &=\left(-c|t|^{-1}+O\left(\left|\log\left|t\right|\right|\right)\right)
                \left(-ci\log\left|t\right|+C'+O\left(\bigl|t\log\left|t\right|\bigr|\right)\right)\\
            &=c^2i|t|^{-1}\log\left|t\right|-cC'|t|^{-1}+O\left((\log\left|t\right|)^2\right).
    \end{align*}
    
    Recalling that $F'(t,n)=\Psi'(t)\sum_{j=1}^{n-1}j\Psi(t)^{j-1}$, we thus have 
    \begin{align*}
        A_3=&-\frac{1}{N^2}\int_{[-\eps,-h]\cup[h,\eps]}e^{-itN}\Th'(t)\psi_Y(t)F'(t,n)\,dt\\
            &\q=-\frac{c^2i}{N^2}\sum_{j=1}^{n-1}j\int_{[-\eps,-h]\cup[h,\eps]}e^{-itN}|t|^{-1}\log\left|t\right|\psi_Y(t)\Psi(t)^{j-1}\,dt\\
        &\q\q-\frac{c C'}{N^2}\sum_{j=1}^{n-1}j\int_{[-\eps,-h]\cup[h,\eps]}e^{-itN}|t|^{-1}\psi_Y(t)\Psi(t)^{j-1}\,dt+O\left(\frac{n^2}{N^2}\right).
    \end{align*}

    Putting  together the expressions for $A_1, A_2, A_3$ and recalling~\eqref{q:I2R1Terms},
    \begin{align*}
        I_2^+(N)
            &=\frac{c^2i}{N^2}\sum_{j=1}^{n-1}\sum_{k=0}^{j-1}\int_{[-\eps,-h]\cup[h,\eps]}e^{-itN}|t|^{-1}\log\left|t\right|\psi_Y(t)\Psi(t)^k\,dt\\
            &\q-\frac{c^2i}{N^2}\sum_{j=1}^{n-1}j\int_{[-\eps,-h]\cup[h,\eps]}e^{-itN}|t|^{-1}\log\left|t\right|\psi_Y(t)\Psi(t)^{j-1}\,dt\\
            &\q-\frac{C_2}{N^2}\sum_{j=1}^{n-1}\sum_{k=0}^{j-1}\int_{[-\eps,-h]\cup[h,\eps]}{e^{-itN}|t|^{-1}\psi_Y(t)\Psi(t)^k}\,dt\\
            &\q-\frac{c C'}{N^2}\sum_{j=1}^{n-1}j\int_{[-\eps,-h]\cup[h,\eps]}e^{-itN}|t|^{-1}\psi_Y(t)\Psi(t)^{j-1}\,dt+O\left(\frac{n^2}{N^2}\right).
    \end{align*}
    
    Recall $h=\pi/N$.
    By Lemma~\ref{lem:tec} with $r=1$, $m=k$ or $m=j-1$ and $M=N$,
    \begin{align*}
     \left|\int_{[-\eps,-h]\cup[h,\eps]}e^{-itN}|t|^{-1}\log\left|t\right|\psi_Y(t)\Psi(t)^m\,dt\right|
        \ll \left|\int_{[\pi/N,\eps]}e^{-itN}t^{-1}\log t\,\psi_Y(t)\Psi(t)^m\,dt\right|=O(\log N).
    \end{align*}
So, the first two sums in the expression of $I_2^+(N)$ are $O((n^2\log N)/N^2)$.
The remaining integrals and sums are treated similarly taking  $r=0$ in Lemma~\ref{lem:tec}.
Thus, $I_2^+(N)=((n^2\log N)/N^2)$.
This together with~\eqref{eq:i20} and~\eqref{eq:i2} gives the result.~\end{proof}

\subsubsection{Estimating the integral $I_3(N)$ defined in~\eqref{q:ITerms}}\label{sec:I3}

Recall that $a_n=n(1+o(1))$.
In this section, we shall use 

\begin{lem}\cite[Lemma 2.3]{MT22}\label{l:MTLem}
    Let $\ell:(0,\infty)\to(0,\infty)$ be a continuous slowly varying function.
    Then for  $n\in\N$ and for any $\rho\geq0$, $$\int_0^\eps t^\rho\ell(t^{-1})|\Psi(t)|^n\,dt=O\big(a_n^{-(\rho+1)}\ell(n)\big)=O\big(n^{-(\rho+1)}\ell(n)\big).$$
\end{lem}
Note that \cite[Lemma 2.3]{MT22} is phrased in terms of $\alpha\ne 1$ stable laws and $\rho>0$, but the proof is word-for-word the same when $\alpha=1$ and $\rho=0$.

\begin{lem}\label{p:I3}
    As $n\to\infty$, $I_3(N)=O(n^2(\log N)^2/N^2)$.
\end{lem}

In this section we will use a modulus-of-continuity argument as in \cite{Katzn} and write
\begin{align}\label{eq:i3}
 I_3(N)=\frac1{iN}\left(\int_h^\eps+\int_{-h}^h+\int_{-\eps}^{-h}\right)e^{-itN}\Th(t)\psi_Y(t)F'(t,n)\,dt= J^+(N)+J^{(0)}(N)+J^-(N).
\end{align}

We start with $J^+(N)$ and $J^-(N)$.
\begin{lem}\label{l:J+}
    As $n\to\infty$, $J^+(N)=O(n^2(\log N)^2/N^2)$.
    A similar estimate holds for $J^-(N)$.
\end{lem}

\begin{proof}We treat the integral $J^+(N)$. The estimate for $J^-(N)$ follows similarly.
    The change of coordinates $t\mapsto t-h$ yields
    $$J^+(N)\defeq\frac1{iN}\int_h^\eps e^{-itN}\Th(t)\psi_Y(t)F'(t,n)\,dt=-\frac1{iN}\int_{2h}^{\eps+h}e^{-itN}\Th(t-h)\psi_Y(t-h)F'(t-h,n)\,dt$$
    and so
    \begin{align}\label{eq:j+}
        \nonumber 2J^+(N)
            &=\frac1{iN}\int_{2h}^\eps e^{-itN}\big(\Th(t)\psi_Y(t)F'(t,n)-\Th(t-h)\psi_Y(t-h)F'(t-h,n)\big)\,dt\\
            \nonumber &\q+\frac1{iN}\int_h^{2h}e^{-itN}\Th(t)\psi_Y(t)F'(t,n)\,dt+\frac1{iN}\int_{\eps-h}^\eps e^{-itN}\Th(t)\psi_Y(t)F'(t,n)\,dt\\
            &\rdefeq J_1(N)+J_2(N)+J_3(N).
    \end{align}    
    
    \underline{The integral $J_1(N)$ in~\eqref{eq:j+}.}\medskip
            
    Compute that
    \begin{align*}
        &\Th(t)\psi_Y(t)F'(t,n)-\Th(t-h)\psi_Y(t-h)F'(t-h,n)\\
        &\q=\big(\Th(t)-\Th(t-h)\big)\psi_Y(t)F'(t,n)+\Th(t-h)\psi_Y(t-h)\big(F'(t,n)-F'(t-h,n)\big)\\
        &\q\q+\Th(t-h)\big(\psi_Y(t)-\psi_Y(t-h)\big)F'(t-h,n)\\
        &\q=\big(\Th(t)-\Th(t-h)\big)\psi_Y(t)F'(t,n)+\Th(t-h)\psi_Y(t-h)\big(F'(t,n)-F'(t-h,n)\big)+O\left(n^2h(\log\left|t\right|)^2\right),
    \end{align*}
    where we have used that $\psi_Y$ is $C^2$ together with~\eqref{q:FEst} and~\eqref{q:ThEst} in the estimate for the big-$O$ term.
    Hence,
    \begin{align}\label{q:J1Terms}
        J_1(N)&=\frac1{iN}\int_{2h}^\eps e^{-itN}\big(\Th(t)\psi_Y(t)F'(t,n)-\Th(t-h)\psi_Y(t-h)F'(t-h,n)\big)\,dt\nonumber\\
            &\,=\frac1{iN}\int_{2h}^\eps e^{-itN}\big(\Th(t)-\Th(t-h)\big)\psi_Y(t)F'(t,n)\,dt\nonumber\\
            &\q+\frac1{iN}\int_{2h}^\eps e^{-itN}\Th(t-h)\psi_Y(t-h)\big(F'(t,n)-F'(t-h,n)\big)\,dt+O\left(\frac{n^2}{N^2}\right)\nonumber\\
            &=B_1(N)+B_2(N)+O\left(\frac{n^2}{N^2}\right).
    \end{align}
    
    We first consider $B_1(N)$ defined in~\eqref{q:J1Terms}.
    Since $h=\frac{\pi}{N}$, using~\eqref{q:F'Est} and~\eqref{q:ThCont}, we obtain
    \begin{align*}
        \big(\Th(t)-\Th(t-h)\big)F'(t,n)=O\left(h\left|t\right|^{-1}\cdot n^2\left|t\log\left|t\right|\right|\right)
            =O\left(\frac{n^2\left|\log\left|t\right|\right|}{N}\right).
    \end{align*}
    Thus,
    \begin{align}\label{eq:b1}
       B_1(N)=O\left(\frac{n^2}{N^2}\int_{2\pi/N}^\eps\left|\log t\right|\, dt\right)=O\left(\frac{n^2}{N^2}\right).
    \end{align}
    
    Next, we consider $B_2(N)$ defined in~\eqref{q:J1Terms}. Using that $h=\pi/N$, together with \eqref{q:F'Cont} and \eqref{q:ThEst}, we find that $N\to\infty$ and $t\to 0$,
    \begin{align*}
        &\Th(t-h)\big(F'(t,n)-F'(t-h,n)\big)\\
        &\q=\left(-c\log\left|t-h\right|+C_2+O\left(\bigl|t\log\left|t\right|\bigr|\right)\right)\left(cih|t|^{-1}(1+o(1))\sum_{j=1}^{n-1}j\Psi(t)^{j-1}
            +O\left(h\sum_{j=1}^{n}j^2\log\left|t\right||\Psi(t)|^{j}\right)\right)\\
        &\q=D_0N^{-1}|t|^{-1}\log\left|t\right|(1+o(1))\sum_{j=1}^{n-1}j\Psi(t)^{j-1}
            +O\left(N^{-1}\sum_{j=1}^{n}j^2(\log\left|t\right|)^2|\Psi(t)|^{j}\right),
    \end{align*}   
    for a complex constant $D_0$.
    Hence, as $N\to\infty$,
    \begin{align*}
    B_2(N)
        =\frac{D_0(1+o(1))}{i N^2}\sum_{j=1}^{n-1}j\int_{2\pi/N}^\eps e^{-itN}t^{-1}\log t\,\Psi(t)^{j-1}\, dt
            +O\left(\frac{1}{N^2}\sum_{j=1}^{n}j^2\int_{0}^\eps(\log t)^2|\Psi(t)|^{j}\, dt\right)
    \end{align*}  
    By Lemma~\ref{l:MTLem} with $\ell(t)=(\log\left|t\right|)^2$ and $\rho=0$,
    \begin{align*}
    \int_{0}^\eps(\log t)^2|\Psi(t)|^{j}\, dt\ll (\log j)^2 j^{-1}.
    \end{align*}  
    which implies that
    $
    \sum_{j=1}^{n}j^2\int_{0}^\eps(\log t)^2|\Psi(t)|^j\, dt\ll  n^2(\log n)^2.
    $
    So,
    \begin{align*}
    B_2(N)=\frac{D_0 (1+o(1))}{i N^2}\sum_{j=1}^{n-1}j\int_{2\pi/N}^\eps e^{-itN}t^{-1}(\log t)^2\Psi(t)^{j-1}\, dt+
    O\left(\frac{n^2(\log n)^2}{N^2}\right).
    \end{align*}
    By Lemma~\ref{lem:tec} with $r=2$, $m=j-1$ and $M=N/2$, $\left|\int_{2\pi/N}^\eps e^{-itN}t^{-1}(\log t)^2\Psi(t)^{j-1}\, dt\right|=O((\log N)^2)$. So,
    \begin{align*}
    B_2(N)=
    O\left(\frac{n^2(\log N)^2}{N^2}\right).
    \end{align*}
    This together with~\eqref{eq:b1} and~\eqref{q:J1Terms} gives $J_1(N)=O(n^2(\log N)^2/N^2)$.\medskip
    
    \underline{The integral $J_2(N)$ and $J_3(N)$ in~\eqref{eq:j+}.}\medskip
    
    The argument is essentially the same. Here we just take advantage of the range
    of integration. 
    We provide the argument for $J_2$. 
    Proceeding directly from the definition,
    $|J_2(N)|\ll \frac{1}{N}\int_{\pi/N}^{2\pi/N}|\Th(t)F'(t,n)|\, dt$.
    By~\eqref{q:F'Est} and~\eqref{q:ThEst}, $|\Th(t)F'(t,n)|\ll n^2 (\log\left|t\right|)^2$.
    Thus, 
    \begin{align*}
    |J_2(N)|\ll \frac{n^2}{N}\int_{\pi/N}^{2\pi/N}(\log t)^2\, dt\ll \frac{n^2 (\log  N)^2}{N^2}.
    \end{align*}
    By a similar argument, $|J_3(N)|\ll n^2/N^2$.
    The estimate for $J^+$ follows by putting together the estimates for $J_1, J_2, J_3$.
\end{proof}

Equipped with Lemma~\ref{l:J+}, we can complete the proof of Lemma~\ref{p:I3}.

\begin{proof}[Proof of Lemma~\ref{p:I3}] Recall equation~\eqref{eq:i3}. 
By Lemma~\ref{l:J+}, $J^+(N)$ and $J^{-}(N)$ are $O(n^2(\log N)^2/N^2)$.
The estimate for $J^{(0)}(N)$ follows by the argument used above in dealing with $J_2(N)$.
Again taking advantage of the range of the integration $[-h, h]=[-\frac{\pi}{N},\frac{\pi}{N}]$,  $|J^{(0)}(N)|\ll n^2 (\log  N)^2/N^2$.
\end{proof}

\subsection{Completing the proof of Theorem~\ref{t:LD}}

We first complete the proof of Proposition~\ref{p:IExp}.

\begin{proof}[Proof of Proposition~\ref{p:IExp}] 
    The conclusion follows from~\eqref{q:ITerms}, Lemma~\ref{l:I1}, Lemma~\ref{l:I2} and Lemma~\ref{p:I3}.
\end{proof}

\begin{proof}[Proof of Theorem~\ref{t:LD}] 
    Recalling~\eqref{q:IDef} and~\eqref{q:INDef}, $I=I(N)-I(N+g(n))$.
    We will use Proposition~\ref{p:gn} or more precisely its consequence Corollary~\ref{cor:I} with $g(n)\ge N^2$.
    By Proposition~\ref{p:IExp}, $I(N)$ and $I(N+g(n))$ are $O(n^2(\log  N)^2/N^2)$
    and $O(n^2(\log  N)^2/g(n)^2)$, respectively. 
    The conclusion follows from Corollary~\ref{cor:I}.
\end{proof}

\end{document}